\documentclass{amsart}

\usepackage[english]{babel}
\usepackage{mathtools, amsmath, amssymb, amsthm}
\usepackage{color, soul, enumitem}
\usepackage{csquotes, dirtytalk}
\usepackage{esint}
\usepackage{orcidlink}

\usepackage{hyperref}   
\usepackage{cleveref}   

\newtheorem{theorem}{Theorem}[section]
\newtheorem{lemma}[theorem]{Lemma}

\theoremstyle{definition}
\newtheorem{definition}[theorem]{Definition}

\newtheorem{remark}[theorem]{Remark}

\crefname{theorem}{theorem}{theorems}
\crefname{lemma}{lemma}{lemmas}
\crefname{corollary}{corollary}{corollaries}
\crefname{proposition}{proposition}{propositions}
\crefname{definition}{definition}{definitions}
\crefname{example}{example}{examples}
\crefname{remark}{remark}{remarks}

\Crefname{theorem}{Theorem}{Theorems}
\Crefname{lemma}{Lemma}{Lemmas}
\Crefname{corollary}{Corollary}{Corollaries}
\Crefname{proposition}{Proposition}{Propositions}
\Crefname{definition}{Definition}{Definitions}
\Crefname{example}{Example}{Examples}
\Crefname{remark}{Remark}{Remarks}

\numberwithin{equation}{section}

\def\elle#1{L^{#1}}
\def\elleom#1{L^{#1}(\Omega)}

\def\Sobom#1{W^{1,#1}_0(\Omega)}
\def\SobomX#1{W^{1,#1}_0(\Omega,X)}

\def\io{\int_\Omega}

\def\norma#1#2{ \|#1 \|_{#2}}

\def\R{\mathbb{R}}

\newcommand{\abs}[1]{\left|#1\right|}

\DeclareRobustCommand{\rchi}{{\mathpalette\irchi\relax}}
\newcommand{\irchi}[2]{\raisebox{\depth}{$#1\chi$}} 

\newlist{assumption}{enumerate}{1}
\setlist[assumption,1]{label=(\roman*),ref=(\roman*)}
\crefname{assumptioni}{assumption}{assumptions}

\title[Solutions to nonlinear X-elliptic equations]{Existence and summability of solutions to nonlinear X-elliptic equations with measurable coefficients}
\author{Marco Picerni
\orcidlink{0009-0004-4364-4831}}\email{mpicerni@sissa.it}
\address{SISSA, via Bonomea 265, 34136, Trieste, Italy}

\begin{document}
	
\maketitle

\section*{Abstract}

We prove an existence result for solutions to a class of nonlinear degenerate-elliptic equations with measurable coefficients and zero Dirichlet boundary condition. The main term is given by a nonlinear operator in divergence form associated to a family of vector fields which satisfy a Poincar\'e inequality and the doubling condition.
Furthermore, we prove that the solutions satisfy a generalization of the $L^p$-regularity results which hold for the solutions to Leray-Lions type equations.\\
\textbf{MSC:} 35H20, 35J60, 35J70, 35B35, 35B45, 35B65, 35R05\\
\textbf{Keywords:} Subelliptic equations, H\"ormander vector fields, Leray-Lions operators, Degenerate elliptic equations, Regularity of solutions

\section{Introduction}

$X$-elliptic operators are a class of second order elliptic operators in divergence form introduced by Kogoj and Lanconelli in \cite{KL-Ricerche} as a generalization of the sub-Laplacian on (sub-Riemannian) Carnot groups. 
The key difference with uniformly elliptic operators lies in the fact that the ellipticity condition
\[\langle Lu,u\rangle\geq \alpha\io\abs{\nabla u}^2\quad\forall u\in C^\infty_c(\Omega),\]
where $\alpha>0$ and $\Omega$ is a bounded domain in $\R^N$ with $N\geq2$,
is relaxed by substituting the gradient term $\nabla u$ with the $X$-gradient $Xu$, which is given by a family of Lipschitz-continuous vector fields $(X_1,\dots,X_d)$, leading by the so-called $X$-ellipticity condition:
\[\langle Lu,u\rangle\geq \alpha\int_\Omega\abs{X u}^2\quad\forall u\in C^\infty_c(\Omega).\]
On the other hand, there is no regularity requirement in the coefficients of the operator, which will often be measurable.

Since then, numerous results have been established for general subelliptic operators and their associated minimization problems (see, for example, \cite{Mazzoni-Green}, \cite{Pinamonti_multiplicity}, \cite{KL-Liouville}, \cite{PallaraMeanValue}, \cite{Subelliptic-DG}).

While the work \cite{Pic-Xell-linear} addresses the regularity properties of solutions to equations involving linear $X$-elliptic operators with measurable coefficients, the aim of this paper is to provide an existence criterion for solutions to equations involving their nonlinear counterpart.

Nonlinear operators in divergence form often fail to satisfy a monotonicity property. Consider, for instance, a Carath\'eodory function $a(x,s,\xi)$ which is monotone in $\xi$ and satisfies suitable growth conditions. In this case, Boccardo and Dacorogna \cite{BoDa_monotonicity} showed that the differential operator
\begin{equation}\label{intro:pseudomon}
    \begin{aligned}
       A:  W^{1,p}_0(\Omega)&\longmapsto W^{-1,p}(\Omega)  \\
         u&\longmapsto -\operatorname{div}(a(x,u,\nabla u)) 
    \end{aligned}
\end{equation}
cannot be monotone unless $a(x,s,\xi)$ does not actually depend on $s$.
Hence, any operator as in \eqref{intro:pseudomon} with a dependence on the function $u$, does not satisfy the assumptions of the Browder–Minty surjectivity theorem, which is used to prove existence results for various elliptic PDEs involving a monotone differential operator.
The Leray-Lions theorem \cite{LL} greatly extends the class of problems to which a solution exists, such as the one in \eqref{intro:pseudomon}. 
In particular, it guarantees the existence (though not necessarily the uniqueness) of solutions to equations where the principal term is merely \emph{pseudomonotone}. This is crucial, as any (reasonable) choice of $a$ in \eqref{intro:pseudomon} induces a pseudomonotone operator.

In \Cref{Sec:existence} of this paper, we prove an extension of the Leray-Lions theorem to the case of nonlinear $X$-elliptic operators.  More precisely, we consider operators of the form 
\[X^* a(x,u,Xu),\]
where $X^*$ denotes the (formal) adjoint of $X$ as an operator from $C^\infty_c(\Omega)$ to $\left(C^\infty_c(\Omega)\right)^d$, $a$ is a Carath\'eodory function which is monotone in its last entry and satisfies suitable growth conditions (which will be stated in \Cref{Sec:existence}). On the right-hand side, we consider a term of the form
\[
f(x)+X^*F(x)
\]
with $f$ and $\abs{F}$ belonging to suitable Lebesgue spaces. We also allow for nonlinearities on the right-hand side, which leads to the problem
\begin{equation}\label{abs:problem}
    \begin{dcases}
        X^* a(x,u,Xu) = F(x,u,X u)+ X^*G(x,u) \quad &\text{in } \Omega\\
        u=0 \quad &\text{on } \partial\Omega.
    \end{dcases}
\end{equation}

Once the existence of a solution is established, we investigate its regularity properties. The works of Stampacchia \cite{StampacchiaElliptic} (in the linear case) and later by Boccardo and Gallou\"et \cite{BoccardoGallouet} (see also the book \cite{BoccardoCroce}) show that, as the summability of the terms on the right-hand side of an elliptic equation increases, the summability of the solutions increases too.
For instance, let $p\in(1,\infty)$ and $\Omega$ be a bounded domain in $\R^N$ and consider the equation
\begin{equation}
    \begin{cases}
        -\Delta_pu=f(x)\quad&\text{in }\Omega\\
        u=0&\text{on }\partial\Omega,
    \end{cases}
\end{equation}
where $\Delta_pu=\operatorname{div}(\abs{\nabla u}^{p-2}\nabla u)$ denotes the $p$-Laplace operator. In this case, there exists a weak solution $u\in\Sobom{p}$ whenever $f$ belongs to $\elleom{m}$ with $m=(p^*)^\prime$ (here, $p^*$ denotes the Sobolev exponent related to $p$). However, for $m>(p^*)^\prime$, it can be shown that $u$ belongs to $\elleom{s}$ with $s=[m^*(p-1)]^*$, which is strictly greater than $p^*$. In \Cref{sec:regularity}, we show that a similar, though possibly weaker, regularizing effect occurs in the $X$-elliptic setting: that is, an increase in the summability of the data leads to an increase in the summability of the solutions. Such improvement depends on the growth rate of the function $a(x,s,\xi)$ and on the doubling constant associated to the family of vector fields $X$. These results generalize the ones found in \cite{Pic-Xell-linear} for linear $X$ elliptic operators. 
Our approach closely follows the techniques developed in \cite{BG-rhs_measures,pLaplacian-BoccardoMerida,BocGiac-regularitynonlinearproblems}.

\section{Preliminaries}\label{Sec:prelim}

Let $\Omega$ be a bounded open subset of $\R^N$, with $N\geq2$, and let $X:=\left\{X_1, \ldots, X_d\right\}$, $d\geq2$, be a family of locally Lipschitz-continuous vector fields on $\Omega$ with coordinate representation $X_j=\left(c_{j 1}, \ldots, c_{j N}\right),\, j=1, \ldots, d$. 
As in \cite{HK-SobPoinc,KL-Liouville}, we identify the vector valued function $c_j$ with the linear first order partial differential operator
$$
X_j=\sum_{k=1}^N c_{j k} \partial_{x_k}.
$$
For any $u\in C^1_c(\Omega)$, we define $X$-gradient of $u$, as
$$
X u=\left(X_1 u, \ldots, X_d u\right) = C(x)\nabla u.
$$
Here $C(x)$ is the $d\times N$ matrix given by $C(x)=\{c_{j k}(x)\}$.

\begin{remark}
    For any function $u\in\elleom1$, the $X$-gradient is defined in the following distributional sense:
    $$\langle X_iu, \varphi\rangle = - \io u X_i\varphi\quad\forall \varphi\in C^1_c(\Omega).$$
    Moreover, $X^*$ denotes the formal adjoint of the differential operator $X$. That is, for any $\varphi\in C^1_c(\Omega)$ and $F$ belonging to $\left(\elleom1\right)^d$, we set
    \[
    \langle X^*F,\varphi\rangle=\langle F,X\varphi\rangle=\io F(x)\cdot X\varphi.
    \]
\end{remark}


We are interested in existence and regularity of solutions to
\begin{equation}\label{Problemabase}
    \begin{dcases}
        X^*a(x,u, X u) = F(x,u,X u) + X^*G(x,u)  \quad &\text{in } \Omega\\
        u=0 \quad &\text{on } \partial\Omega
    \end{dcases}
\end{equation}
where $a: \Omega \times \mathbb{R} \times \mathbb{R}^d \rightarrow \mathbb{R}^d$, is a Carath\'eodory function such that the following holds for almost every $x \in \Omega$, for every $s \in \mathbb{R}$ and for every $\xi \neq \eta \in \mathbb{R}^d$.
$$
\left\{\begin{array}{l}
    a(x, s, \xi) \xi \geq \alpha|\xi|^p \\
    |a(x, s, \xi)| \leq \beta(h(x)+\abs{s}^{p-1}+|\xi|^{p-1}) \\
    {[a(x, s, \xi)-a(x, s, \eta)](\xi-\eta)>0,}
\end{array}\right.
$$
where $\alpha, \beta$ are positive constants, $1<p<\infty$ and $h\in\elleom{p'}$. We also assume that  $$F: \Omega \times \mathbb{R} \times \mathbb{R}^d \rightarrow \mathbb{R}$$
$$G: \Omega \times \mathbb{R} \rightarrow \mathbb{R}^d$$
and are two Carath\'eodory functions such that $\abs{F(x,s,\xi)}\leq f(x)$ for some $f\in\elleom r$, $r\geq1$ and $\abs{G(x,s)}\leq g(x)$ for some $g\in\elleom q$, $q\geq 1$ 
(the values of $r$ and $q$ will be specified in \Cref{Sec:existence}).

\subsection{Geometric assumptions}

An absolutely continuous path $\gamma:[0, T] \rightarrow \Omega$ is said $X$-subunit if $\dot{\gamma}(t)=\sum_{j=1}^m \theta_j(t) X_j(\gamma(t))$, with $\sum_{j=1}^m \theta_j^2(t) \leq 1$, for almost every $t \in[0, T]$. Assuming that $\Omega$ is $X$-connected, i.e. for every $x, y \in \Omega$ there exists at least one $X$-subunit path connecting $x$ and $y$, we define 
$$d_X(x, y):=\inf \left\{T>0 \mid \exists \gamma:[0, T] \rightarrow \Omega\,  \text{ $X$-subunit such that } \gamma(0)=x, \gamma(T)=y\right\}.$$

We will make the following assumptions on the family of vector fields $X$, which have been thoroughly outlined in \cite{Mazzoni-Green}, \cite{KL-Liouville} and \cite{HK-SobPoinc}.

\begin{itemize}
    \item The Carnot-Carath\'eodory metric $\delta$ is equivalent to the euclidean norm 
    $$|y-x| \rightarrow 0 \quad \iff\quad \delta(x, y) \rightarrow 0.$$ This is true, for example, if the vector fields $X_j$ are smooth and satisfy the H\"ormander condition (also known as Lie Algebra Rank Condition), which, in this case, means that the commutators of the vector fields $X_i$ up to some length $s\leq N$ span $\R^N$.
    \item $\R^N$, equipped with the control distance and the Lebesgue measure, is a doubling space, i.e. there exists a constant $Q>0$ such that 
    \begin{equation}\label{QDoubling}
        0<\abs{B_{2r}}\leq2^Q\abs{B_r},
    \end{equation}
    where $B_r$ is the $\delta$-ball of radius $r$. Note that we can assume, without loss of generality, that $Q>2$.
    \item There exist positive constants $C_P, \nu$ such that the following $1,1$-Poincar\'e inequality holds
    \begin{equation}\label{PoincaINEQ}
        \fint_{B_r}\left|u-u_r\right| \mathrm{d} x \leq C_P r \fint_{B_{\nu r}}|X u| \mathrm{d} x, \quad \forall u \in C^1\left(\overline{B_{\nu r}}\right)
    \end{equation}
    for any $\delta$-ball $B_r$, with $u_r=:=\frac{1}{\left|B_r\right|} \int_{B_r} u \mathrm{~d} x$.
\end{itemize}


\subsection{Functional spaces}

Following \cite{HK-SobPoinc}, we define the Sobolev spaces associated to the family of vector fields $X$.

\begin{definition}
    We define the space $W^{1,p}_0(\Omega,X)$ as the closure of $C^\infty_c(\Omega)$ under the norm
    $$\norma{u}{W^{1,p}_0(\Omega,X)}=\norma{Xu}{\elleom p}.$$
\end{definition}

Here we recall some properties of these spaces. For a more thorough discussion, we refer the reader to \cite{HK-SobPoinc,KL-Liouville}.
First of all, we have the following equivalent of the Sobolev embedding theorem.
\begin{theorem}\label{SobolevINEQ}(Sobolev Inequality). Let $1\leq p<Q$, with $Q$ given by \eqref{QDoubling}. There exists a positive constant $\mathcal{S}_p$ (only depending on $C_P$, $\nu$, $Q$ and $p$) such that
    \begin{equation}\label{eq:SobolevINEQ}        
    \norma{u}{\elleom {p^*}}\leq \mathcal{S}_p\norma{Xu}{\elleom p} \quad \forall u \in W^{1,p}_0(\Omega,X),\quad p^*=\frac{p Q}{Q-p} .
    \end{equation}
\end{theorem}

\begin{remark}
    We use the notation $p^*$, as in the classical case, to denote the Sobolev exponent. This choice is to highlight the fact that the Sobolev spaces only "see" the dimension $Q$ given by the doubling condition \eqref{QDoubling}, rather than $N$. This will have a crucial effect on the regularity of solutions to \eqref{Problemabase}.
    We also define $p_*$ as the H\"older conjugate of $p^*$, namely 
    $$p_*=\frac{pQ}{Q(p-1)+p}=\frac{Qp'}{Q+p'}.$$
\end{remark}
\begin{remark}\label{pstarinbasso}
    $\elleom{p_*}$ is a subspace of $\left(W^{1,p}_0(\Omega,X)\right)^\prime$.
\end{remark}
The Rellich-Kondrakov theorem extends to this setting as follows.
\begin{theorem}\label{RellichTH}
    If $1\leq p<Q$ and $q<p^*=\frac{Qp}{Q-p}$, any bounded sequence $(u_n)_n\subset\SobomX p$ strongly converges, up to subsequences, in $\elleom q$.
\end{theorem}
We also have an equivalent of the Trudinger inequality (see \cite{HK-SobPoinc}).
\begin{theorem}(Trudinger inequality)\label{TrudingerINEQ}
    There exist two positive contants $C_1$ and $C_T$ such that, for every $u\in\SobomX Q$,
    $$\io e^{\left(C_1\frac{\abs{u}}{\norma{Xu}{\elleom Q}}\right)^\frac{Q}{Q-1}}\leq C_T.$$
\end{theorem}
If $p>Q$, a Morrey-type embedding theorem is in place (see \cite{HK-SobPoinc}).
\begin{theorem}\label{MorreyTH}
    If $u\in\SobomX p$ for some $p>Q$, then $u$ is (after redefinition on a set of measure zero) locally H\"older continuous. Moreover, we have the estimate
    $$
    |u(x)-u(y)| \leq C r_0^{Q / p} d(x, y)^{1-Q / p}\left(\int_{5 \sigma B_{r_0}} \abs{Xu}^p d \mu\right)^{\frac{1}{p}}
    $$
    for all $x, y \in B_{r_0}$, where $B_{r_0}\subset \Omega$ is a $d_X$-ball of radius $r_0$.
\end{theorem}
\begin{remark}\label{Remark_AA}
    Since the space $\SobomX p$ is given by the closure of $C^\infty_c(\Omega)$ under the Sobolev norm, every function belonging to $\SobomX p$ with $p>Q$ is H\"older continuous and zero on the boundary of $\Omega$ in the classical sense. 
    Moreover, since the H\"older seminorm is controlled by the $\SobomX p$ norm, any bounded sequence in $\SobomX p$ has a converging subsequence in $\elleom\infty$ by the Ascoli-Arzelà theorem.
\end{remark}

The following result follows from the classical chain rule for Sobolev functions.

\begin{lemma}
    If $\psi$ is a Lipschitz-continuous function on $\R$ such that $\psi(0)=0$ and $u\in\SobomX p$, then $\psi(u)\in\SobomX p$ and $X\psi(u)=\psi'(u)Xu$.
\end{lemma}
\begin{remark}\label{Rem:TkGk}
    We introduce the truncation function $T_k$ for $k>0$:
    \begin{equation}\label{eq_def:Tk}
        T_k(s) = \begin{dcases}
            s \quad & \abs s\leq k\\
            k \quad & s> k\\
            -k \quad & s<- k\\
        \end{dcases}
    \end{equation}
    and define
    \begin{equation}\label{eq_def:Gk}
        G_k(s)=s-T_k(s)=(\abs s-k)_+\rm{sgn}(s)
    \end{equation}
    Then, by the previous lemma, we have
    $$X(T_k(u))=Xu\rchi_{\{\abs{u}\leq k\}}\quad\text{and}\quad X(G_k(u))=Xu\rchi_{\{\abs{u}\geq k\}.}$$
\end{remark}

We conclude this section with the following characterization theorem for the dual of $\SobomX p$, which can be proved by following the classical proof outlined in \cite{Brezis}.
\begin{theorem}
    Let $1\leq p<\infty$ and $\Lambda\in\left(\SobomX p\right)^\prime$. There exists $F\in\left(\elleom{p'}\right)^d$, where $d$ is the number of components of $X$, such that
    $$\langle \Lambda, u\rangle= \io F(x)\cdot Xu=\langle X^*F,u\rangle\quad\forall u\in\SobomX{p}.$$
    Moreover, if $1<p<\infty$, $\SobomX p$ is reflexive.
\end{theorem}

\section{Existence of a weak solution}\label{Sec:existence}

In the uniformly elliptic case (that is, when $X$ is the standard gradient), the Leray-Lions theorem \cite{LL} establishes a sufficient condition for the existence of a weak solution to differential equations in divergence form with principal term of the type \eqref{intro:pseudomon} and zero Dirichlet boundary conditions. That is, a function $u\in\Sobom{p}$ such that
$$\io a(x,u,\nabla u)\cdot\nabla v = \io F(x,u,\nabla u) v\quad\forall v\in\Sobom{p},$$
where $a$ and $F$ satisfy certain growth conditions.
The following results extends the Leray-Lions theorem to $X$-elliptic nonlinear problems such as \eqref{Problemabase}.

\begin{theorem}\label{LL-Xelliptic}
    Let $1<p<\infty$ and assume that $$a: \Omega \times \mathbb{R} \times \mathbb{R}^d \rightarrow \mathbb{R}^d$$
    $$F: \Omega \times \mathbb{R} \times \mathbb{R}^d \rightarrow \mathbb{R}$$
    $$G: \Omega \times \mathbb{R} \rightarrow \mathbb{R}^d$$
    are Carath\'eodory functions such that the following conditions hold for almost every $x \in \Omega$, for every $s \in \mathbb{R}$ and every $\xi \neq \eta \in \mathbb{R}^d$.
    \begin{assumption}
        \item \label{ass:ellipticity} $a(x, s, \xi) \xi \geq \alpha|\xi|^p$;
        \item \label{ass:grow-a} $|a(x, s, \xi)| \leq \beta(h(x)+\abs{s}^{p-1}+|\xi|^{p-1})$;
        \item \label{ass:monot} ${[a(x, s, \xi)-a(x, s, \eta)]\cdot(\xi-\eta)>0}$;
        \item \label{ass:grow-F} $\abs{F(x,s,\xi)}\leq f(x)$;
        \item \label{ass:grow-G} $\abs{G(x,s)}\leq g(x)$;
    \end{assumption}
    where $\alpha, \beta$ are positive constants, $h\in\elleom{p'}$, $r>p_*$ (if $p\leq Q$, otherwise $r=1$), $f\in\elleom{r}$, and $g\in\elleom{p'}$. 
    Then there exists a weak solution $u\in\SobomX{p}$ to \eqref{Problemabase}, namely:
    $$\io a(x,u,Xu)\cdot Xv=\io F(x,u,Xu)v +\io G(x,u)\cdot Xv \quad\forall v\in \SobomX{p}.$$
\end{theorem}

    To prove Theorem \ref{LL-Xelliptic}, as for the classical case, we show that the operator
    \begin{equation}\label{DefOpLL}
        \begin{aligned}
       A:  \SobomX p&\longmapsto \left(\SobomX p\right)^\prime  \\
         u&\longmapsto X^* a(x,u,Xu)-F(x,u,X u)-X^*G(x,u)
        \end{aligned}
    \end{equation}
    is coercive and pseudomonotone. To that end, we need the following two lemmas (whose proofs are straightforward adaptations of the classical case, see e.g. \cite{BoccardoCroce}). 

    \begin{lemma}\label{LemmaLL1}
        Under the assumptions of Theorem \ref{LL-Xelliptic}, let $(u_n)_n$ be a sequence in $\SobomX p$ such that:
        \begin{itemize}
            \item $u_n\rightharpoonup u$ in $\SobomX p$;
            \item  $[a(x,u_n,Xu_n)-a(x,u,Xu)]\cdot X(u_n-u)\to0$ a.e. in $\Omega$.
        \end{itemize}
        Then $Xu_n\to Xu$ a.e. in $\Omega$.
    \end{lemma}
    
\begin{proof}
    Let $Z$ be a set of measure zero such that
    $$[a(x,u_n,Xu_n)-a(x,u,Xu)]\cdot X(u_n-u)\to0\quad\forall x\in\Omega\setminus Z$$
    and assume that $u_n(x)\to u(x)$ for every $x\in\Omega\setminus Z$.
    Now let $x_0\in\Omega\setminus Z$ and note that the sequence $(Xu_n(x_0))_n$ is bounded. Indeed, we have
    $$[a(x_0,u_n(x_0),Xu_n(x_0))-a(x_0,u(x_0),Xu(x_0))]\cdot (Xu_n(x_0)-Xu(x_0))\to0$$
    which implies, using the properties of $a(x,s,\xi)$, that
    \begin{equation}
        \begin{split}
            \alpha \abs{Xu_n(x_0)}^p
            \leq& a(x_0,u_n(x_0),Xu_n(x_0))\cdot Xu_n(x_0)\\
            \leq &1+\abs{a(x_0,u(x_0),Xu(x_0))}\abs{ Xu_n(x_0)} \\&+ \abs{a(x_0,u(x_0),Xu(x_0))}\abs{ Xu(x_0)}\\&+ \abs{a(x_0,u_n(x_0),Xu_n(x_0))}\abs{Xu(x_0)}\\
            \leq& C(1+\abs{ Xu_n(x_0)}^{p-1} + \abs{ Xu_n(x_0)})
        \end{split}
    \end{equation}    
    Let $b\in\R^d$ be a limit point of the sequence $(Xu_n(x_0))_n$. Taking the limit as $n\to\infty$ leads to
    $$[a(x_0,u(x_0),b)-a(x_0,u(x_0),Xu(x_0))]\cdot (b-Xu(x_0))=0,$$
    which, by monotonicity of $a(x,s,\xi)$ in the last variable, implies that $b=Xu(x_0)$, thereby concluding the proof.
\end{proof}

    \begin{lemma}\label{LemmaLL2}
        Under the assumptions of \Cref{LL-Xelliptic}, let $(u_n)_n$ be a sequence in $\SobomX p$ such that 
        \begin{itemize}
            \item $u_n\rightharpoonup u$ in $\SobomX p$;
            \item $\displaystyle\io [a(x,u_n,Xu_n)-a(x,u,Xu)]\cdot X(u_n-u)\to0.$
        \end{itemize} 
        Then 
        \begin{equation}
            a(x,u_n,Xu_n)\rightharpoonup a(x,u,Xu)\quad\text{ in }\elleom {p'}
        \end{equation}
        and 
        \begin{equation}
            [a(x,u_n,Xu_n)-a(x,u,Xu)]\cdot X(u_n-u)\to0\text{ a.e. in $\Omega$}.
        \end{equation} 
    \end{lemma}

\begin{proof}
    We add and subtract $a(x,u_n,Xu)\cdot X(u_n-u)$ from 
    $$\io [a(x,u_n,Xu_n)-a(x,u,Xu)]\cdot X(u_n-u)\to0,$$
     to obtain
     \begin{equation}
         \begin{split}
             \io [a(x,u_n,Xu_n)-&a(x,u_n,Xu)]\cdot X(u_n-u) \\&+\io [a(x,u_n,Xu)-a(x,u,Xu)]\cdot X(u_n-u)\to0.
         \end{split}
     \end{equation}
    Observe that the second integral tends to zero. This is because $a(x,u_n,Xu)$ converges to $a(x,u,Xu)$ in $\elleom {p'}$ by the Nemytskii composition theorem and $X(u_n-u)\rightharpoonup 0$ in $\elleom p$. It follows that
    \begin{equation}
        \begin{split}
            \io&\abs{[a(x,u_n,Xu)-a(x,u,Xu)]\cdot X(u_n-u)}\\
            &\quad\quad\leq\norma{a(x,u_n,Xu)-a(x,u,Xu)}{\elleom {p'}}\norma{X(u_n-u)}{\elleom p}\to 0.
        \end{split}
    \end{equation}
    This implies that the first term, which is positive by monotonicity of $a(x,s,\xi)$, tends to zero. Thus  
    $$[a(x,u_n,Xu_n)-a(x,u,Xu)]\cdot X(u_n-u)\to0\quad\text{ in }\elleom1,$$
    which, by Lemma \ref{LemmaLL1}, implies that $Xu_n\to Xu$ almost everywhere. Moreover, since $(a(x,u_n,Xu_n))_n$ is a bounded sequence in $\elleom {p'}$ which converges almost everywhere to $a(x,u,Xu)$, it must also weakly converge to such limit. This concludes the proof.
\end{proof}

    We now prove Theorem \ref{LL-Xelliptic}
\begin{proof}
    We begin by proving that the operator $A$ defined in \eqref{DefOpLL} is coercive: indeed, one has
    \begin{equation*}
        \begin{split}
            \langle A(v),v\rangle
            &=\io a(x,v,Xv)\cdot Xv - \io F(x,v,Xv)v -\io G(x,v)\cdot Xv\\
            &\geq\alpha\io \abs{Xv}^p-\norma{f}{\elleom {r}}\norma{v}{\elleom {r'}}-\norma{g}{\elleom {p'}}\norma{Xv}{\elleom p}\\
            &\geq \alpha\norma{v}{\SobomX p}^p-\tilde C\norma{f}{\elleom {r}}\norma{v}{\SobomX p}-\norma{g}{\elleom {p'}}\norma{v}{\SobomX p}
        \end{split}
    \end{equation*}
    where the last step follows from \Cref{SobolevINEQ} if $p\leq Q$ (resp. \Cref{MorreyTH}, if $p>Q$).
    
    We now show that $A$ is bounded. Let $v,w\in\SobomX p$ and note that, by the properties of $a$, $F$, $G$ and the H\"older inequality:
    \begin{equation}
        \begin{split}
            \langle A(v),w\rangle & = \io a(x,v,Xv)\cdot Xw - \io F(x,v,Xv)w -\io G(x,v)\cdot Xw\\
            & \begin{split}
                \leq \beta\io h(x)\abs{Xw}+\abs{v}^{p-1}\abs{Xw}+|Xv|^{p-1}\abs{Xw}
                \\+\io f(x)\abs{w}+\io g(x)\abs{Xw}
            \end{split}\\
            &\begin{split}
                \leq\beta\left(\norma{h}{\elleom {p'}}+\norma{v}{\elleom p}^{p-1}+\norma{Xv}{\elleom p}^{p-1}\right)\norma{w}{\SobomX p}
                \\+\norma{f}{\elleom {r}}\norma{w}{\elleom {r'}}+\norma{g}{\elleom {p'}}\norma{Xw}{\elleom {p}},
            \end{split}
        \end{split}
    \end{equation}
    which implies
    \begin{equation}
        \langle A(v),w\rangle\leq C\left(1+\norma{v}{\SobomX p} ^{p-1}\right)\norma{w}{\SobomX p}
    \end{equation}
    for some positive constant $C$ depending on the data.
    
    We now prove the pseudomonotonicity of the operator $A$. Let $(u_n)_n\subset\SobomX p$ be a sequence such that
    \begin{itemize}
        \item $u_n\rightharpoonup u$ in $\SobomX p$;
        \item $\limsup_{n\to\infty}\langle A(u_n),u_n-u\rangle\leq 0$.
    \end{itemize}
    We claim that
    \begin{equation}\label{ObiettivoProofLL}
        \liminf_{n\to\infty} \langle A(u_n),u_n-w\rangle\geq \langle A(u),u-w\rangle\quad\forall w\in\SobomX p.
    \end{equation}
    Note that, if $p\leq Q$, then $u_n\to u$ in $\elleom q$ for every $p\leq q<p^*$, by \Cref{RellichTH}. On the other hand, if $p>Q$, then $u_n\to u$ in $\elleom\infty$ by Remark \ref{Remark_AA}.
    In both cases, we have
    \begin{equation}\label{proof:LL:convr'}
        u_n\to u\quad\text{in }\elleom{r'}.
    \end{equation}    
    We expand the left-hand side of \eqref{ObiettivoProofLL} to get
    \begin{equation}
        \begin{split}
            \langle A(u_n),u_n-w\rangle &=\io a(x,u_n,Xu_n)\cdot X(u_n-w) \\&- \io F(x,u_n,Xu_n)(u_n-w)-\io G(x,u_n)\cdot X(u_n-w).
        \end{split}
    \end{equation}
    First, we prove that
    \begin{equation}
        \lim_{n\to\infty} \io [a(x,u_n,Xu_n)-a(x,u,Xu)]\cdot X(u_n-u)=0.
    \end{equation}
    To this end, note that, due to \Cref{ass:grow-F,ass:grow-G},
    the sequence $\big(F(x,u_n,Xu_n)\big)_n$ is bounded in $\elleom {r}$ and $G(x,u_n)\to G(x,u)$ in $\left(\elleom{p'}\right)^d$. 
    This, by \eqref{proof:LL:convr'} and the weak convergence of $Xu_n$ to $Xu$ in $\elleom p$, leads to
    $$\abs{\io F(x,u_n,Xu_n)(u_n-u)}\leq \norma{F(x,u_n,Xu_n)}{\elleom {r}}\norma{u_n-u}{\elleom {r'}}\to0$$
    and
    $$\abs{\io G(x,u_n)\cdot X(u_n-u)}\to0.$$
    This, along with the assumption that $\limsup_{n\to\infty}\langle A(u_n),u_n-u\rangle\leq 0$, implies
    $$\limsup_{n\to\infty} \io [a(x,u_n,Xu_n)-a(x,u,Xu)]\cdot X(u_n-u)\leq0.$$
    On the other hand, by the monotonicity of $a(x,s,\xi)$ in the last variable,
    \begin{equation*}
        \begin{split}
             \io [a(x,u_n,Xu_n)-&a(x,u,Xu)]\cdot X(u_n-u)\\&\geq \io [a(x,u_n,Xu)-a(x,u,Xu)]\cdot X(u_n-u)\to 0,
        \end{split}
    \end{equation*}
    where the limit follows from the Nemytskii composition theorem for the operator $v\mapsto a(x,v,Xu)$.
    Next, we apply Lemma \ref{LemmaLL2} to establish that 
    \begin{equation}\label{proof:LL:weakconv}
        a(x,u_n,Xu_n)\rightharpoonup a(x,u,Xu)\text{ in }\elleom {p'}
    \end{equation}
    and 
    $$a(x,u_n,Xu_n)\to a(x,u,Xu)\text{ a.e. in }\Omega,$$
    which allows us to conclude, by Lemma \ref{LemmaLL1}, that $Xu_n\to Xu$ a.e. in $\Omega$. 
    In turn, an application of Fatou's Lemma yields
    $$\liminf_{n\to\infty}\io a(x,u_n,Xu_n)\cdot Xu_n\geq \io a(x,u,Xu)\cdot Xu$$
    and, by \eqref{proof:LL:weakconv}, we get
    $$\lim_{n\to\infty}\io a(x,u_n,Xu_n)\cdot Xw= \io a(x,u,Xu)\cdot Xw\quad\forall w\in\SobomX p,$$
    which leads to
    \begin{equation}\label{proof:LL:pseudomon1}
        \liminf_{n\to\infty}\io a(x,u_n,Xu_n)\cdot X(u_n-w)\geq \io a(x,u,Xu)\cdot X(u-w).
    \end{equation}
   
    We now examine the term
    $$\io F(x,u_n,Xu_n)(u_n-w).$$
    By the pointwise convergence of $u_n$ and $Xu_n$ to, respectively, $u$ and $Xu$, we know that $F(x,u_n,Xu_n)\to F(x,u,Xu)$ almost-everywhere. Additionally, since the sequence $(F(x,u_n,Xu_n))_n$ is bounded in $\elleom {r}$ (note that, if $r=1$, the sequence is precompact by the Dunford-Pettis theorem), it weakly converges to $F(x,u,Xu)$ in $\elleom {r}$.
    It follows that
    \begin{equation}\label{proof:LL:pseudomon2}
        \io F(x,u_n,Xu_n)(u_n-w) \to \io F(x,u,Xu)(u-w).
    \end{equation}

    As for the term
    $$\io G(x,u_n)\cdot X(u_n-w),$$
    it suffices to observe that $G(x,u_n)\to G(x,u)$ in $\elleom{p'}$ by Lebesgue's Theorem and that $X(u_n-w)\rightharpoonup X(u-w)$ in $\elleom p$ to conclude that
    \begin{equation}\label{proof:LL:pseudomon3}
        \io G(x,u_n)\cdot X(u_n-w) \to \io G(x,u)\cdot X(u-w).
    \end{equation}

    Putting \eqref{proof:LL:pseudomon1}, \eqref{proof:LL:pseudomon2} and \eqref{proof:LL:pseudomon3} together, we have
    $$\liminf_{n\to\infty} \langle A(u_n),u_n-w\rangle\geq \langle A(u),u-w\rangle.$$
    This proves that $A$ is a pseudomonotone operator, thereby concluding the proof of \Cref{LL-Xelliptic}.
\end{proof}

\begin{remark}
    \Cref{ass:grow-F} may be slightly relaxed: indeed, we just need the coercivity of the operator $A$ and the boundedness in $\elleom r$ (and uniform integrability, if $r=1$) of the sequence $(F(x,u_n, Xu_n))_n$. For example, if $p\leq Q$, these properties are in place under the assumption
    \begin{equation}
        \abs{F(x,s,\xi)}\leq C(f(x)+\abs{s}^\kappa+\abs{\xi}^\sigma)
    \end{equation}
    for some $\kappa<p^*$ and $\sigma<p$.
    
    We could also relax 
    \Cref{ass:grow-G}: in particular, as long as the Nemytskii operator 
    \[v\mapsto G(x,v)\]
    is continuous from $\elleom {\eta}$ to $\elleom {p'}$ for some $\eta<p^*$ (or, if $p>Q$, from $\elleom\infty$ to $\elleom {p'}$) and  
    \begin{equation}
        \norma{G(x,v)}{\elleom{p'}}\leq C\left(1+\norma{v}{\SobomX p}^\tau\right)
    \end{equation}
    for some $\tau<p$ and $C>0$, the proof of \Cref{LL-Xelliptic} works analogously (up to using the Nemytskii composition theorem to prove that $G(x,u_n)\to G(x,u)$ in $\elleom {p'}$).
\end{remark}

We conclude this section with two remarks. Recall that the classical case (that is, when $X$ is the standard gradient) is a particular case of the context we are studying.

\begin{remark}
    In \cite{BoDa_pseudomon} it was shown that, in the classical case, the monotonicity condition on $a(x,s,\xi)$ (that is, \Cref{ass:monot}) is a necessary assumption for the pseudomonotonicity of the operator $A$. Thus, Theorem \ref{LL-Xelliptic} can be considered optimal in this regard.
\end{remark}

\begin{remark}
    Theorem \ref{LL-Xelliptic} establishes the pseudomonotonicity (and not the monotonicity) of the operator $A$. It follows that, in general, we only expect the existence of a weak solution to \eqref{Problemabase}: without further assumptions, the uniqueness of such a solution (which usually comes from the monotonicity of the operator) should not be expected. 
    In fact, as  shown in \cite{BoDa_monotonicity} a Leray-Lions operator of the form
    $$A(u)=-div(b(u)\nabla u))$$
    is monotone if and only if the function $b$ is constant. 
    
    On the other hand, when considering a Carath\'eodory function $a(x,\xi)$ which satisfies the assumptions of Theorem \ref{LL-Xelliptic}, the resulting operator
    $$A(u)=X^*(a(x,Xu))$$
    is a monotone, bounded, continuous and coercive map from $\SobomX p$ to its dual. In this case, the Browder-Minty theorem implies that $A$ is a bijective map from, $\SobomX p$ to its dual. This proves both the existence and uniqueness of a solution to problems of the type
    $$
    \begin{dcases}
        X^*a(x, X u) = f(x)+X^*F(x) \quad &\text{in } \Omega\\
        u=0 \quad &\text{on } \partial\Omega
    \end{dcases}$$
    provided that $f(x)+X^*F(x)$ belongs to $\left(\SobomX p\right)^\prime$. 
\end{remark}

\section{Summability of the solutions}\label{sec:regularity}

In this section we prove that, as the summability of the terms on the right-hand side of \eqref{Problemabase} increases, the summability of any solution $u$, which exists by Theorem \ref{LL-Xelliptic}, increases too.
Moreover, since the nonlinearities on the right-hand side of \eqref{Problemabase}, in view of \Cref{ass:grow-F,ass:grow-G}, do not impact the summability of the corresponding terms, we consider the model problem
\begin{equation}
    \begin{dcases}
        X^*a(x,u, X u) = f(x)+X^*F(x) \quad &\text{in } \Omega\\
        u=0 \quad &\text{on } \partial\Omega.
    \end{dcases}
\end{equation}

Although the problem is nonlinear, we consider separately the case in which the right-hand side is a measurable function and the case in which it is a distribution in divergence form. This distinction, while not affecting the techniques used to establish the regularity results (since they can be applied to both terms simultaneously) allows us to simplify the computations and clarifies the role played by the summability of $f$ and $F$.

First, consider the case $F=0$, leading to the problem
\begin{equation}\label{Problema_solof}
    \begin{dcases}
        X^*a(x,u, X u) = f(x) \quad &\text{in } \Omega\\
        u=0 \quad &\text{on } \partial\Omega,
    \end{dcases}
\end{equation}
where $f\in\elleom{m}$, $m\geq p_*$. We claim that, as $m$ increases, the summability of $u$ increases too.

\begin{remark}
    Under these summability assumptions for $f$ and no regularity assumptions on $a(x,s,\xi)$ in the first variable, one should not expect a systematic improvement in the summability of the $X$-gradient of $u$. In the uniformly elliptic case, it was shown by Meyers in \cite{Meyers} that an increase in summability does indeed occur, but it depends on the ellipticity constant of the operator.
    On the other hand, it was proven in \cite{MingioneHeisenberg} that, if $a$ is a $C^1$ function depending only on $\xi$, a regularizing effect on $Xu$ shall be expected.
\end{remark}

\begin{remark}
    Since, for $p>Q$, the space $\SobomX{p}$ is embedded in $C^0_0(\Omega)$ (the space continuous functions with zero boundary value), which is a subspace of $\elleom{\infty}$, we will restrict our study to the case $p\leq Q$.
\end{remark}

\subsection{The case \texorpdfstring{$p<Q$}{p<Q}}

We prove the following two results, which extend the ones proved in \cite{Pic-Xell-linear} for linear $X$-elliptic operators. The first deals with highly summable data (namely, $f\in\elleom m$ with $m>\frac{Q}{p}$) and follows the technique from \cite{StampacchiaElliptic}.
\begin{theorem}\label{Theo:u_limitata}
    Let $f\in\elleom m$, with $m>\frac{Q}{p}$, and let $u$ be a weak solution to \eqref{Problema_solof}. Then $u\in\elleom\infty$ and there exists a constant $C=C(Q,p,\Omega,\alpha,m,X)$ such that
    $$
    \norma{u}{\elleom\infty} \leq C \norma f{\elleom m}^\frac{1}{p-1}.
    $$ 
\end{theorem}

If we consider data belonging to the dual of $\SobomX p$, but not satisfying the assumptions of \Cref{Theo:u_limitata} (that is, if $p_*\leq m < \frac{Q}{p}$), we can still expect an improvement in the summability of the solution $u$, as the next theorem (which follows the techniques from \cite{BoccardoGallouet}) shows. 
\begin{theorem}\label{theo:Stampacchia-p**}
    Let $f\in\elleom m$, with $p_*=\frac{pQ}{Q(p-1)+p} \leq m<\frac{Q}{p}$ and let $u$ be a weak solution to \eqref{Problema_solof}. Then $u$ belongs to $\elleom {s}$, with 
    \[s=\frac{m(p-1)Q}{Q-m p}=[m^*(p-1)]^*,\]
    and there exists a constant $C=C(Q,p,\Omega,\alpha,m,X)$ such that
    \begin{equation}
    \norma{u}{\elleom s} \leq C \norma f{\elleom m}^\frac{1}{p-1}.
    \end{equation} 
\end{theorem}

\begin{remark}
    Note that $\frac{Q}{p}$ is strictly greater than $p_*$ as long as $1<p<Q$.
\end{remark}
\begin{remark}
    The constant $C$ in Theorems \ref{Theo:u_limitata} and \ref{theo:Stampacchia-p**} (see also \Cref{Theo:u_limitata_div} and \Cref{theo:Stampacchia-p**_div}) depends on the domain $\Omega$ and the family of vector fields $X$ through the Sobolev constant $\mathcal{S}_p$ appearing in \eqref{eq:SobolevINEQ}.
\end{remark}

To prove the first result,  we will use the following lemma from \cite{StampacchiaElliptic}.

\begin{lemma}\label{LemmaStampacchiaLinfty}
    Let $\psi: \mathbb{R}^{+} \rightarrow \mathbb{R}^{+}$be a nonincreasing function such that
    $$
    \psi(h) \leq \frac{M \psi(k)^\delta}{(h-k)^\gamma}, \quad \forall\, h>k \geq 0,
    $$
    for some $M>0, \delta>1$ and $\gamma>0$. Then there exists $d\in\R$ satisfying
    $$
    d^\gamma=M \psi(0)^{\delta-1} 2^{\frac{\delta \gamma}{\delta-1}},
    $$
    such that $\psi(d)=0$.
\end{lemma}

\begin{proof}[Proof of Theorem \ref{Theo:u_limitata}] Here, we follow \cite{StampacchiaElliptic}.
    Let $k \geq 0$ and define $A_k=\{x \in \Omega:|u(x)| \geq k\}$.
    We choose $v=G_k(u)$ (where $G_k$ is the function defined in \eqref{eq_def:Gk}) as test function in \eqref{Problemabase}. By Remark \ref{Rem:TkGk}, \Cref{ass:ellipticity}, and since $G_k(u)=0$ outside of $A_k$, we have
    $$
    \alpha \int_{A_k} \left| X G_k(u)\right|^p 
    \leq 
    \io a(x,u,X u)\cdot  Xu \rchi_{A_k}
     = \io f G_k(u)
    =
    \int_{A_k} f G_k(u).
    $$
    Using the \Cref{SobolevINEQ} (on the left-hand side) and the H\"older inequality (on the right-hand side), we have
    \begin{equation}
        \frac{\alpha}{\mathcal{S}_p^p}\left(\int_{A_k}\left|G_k(u)\right|^{p^*}\right)^{\frac{p}{p^ *}} \leq
    \left(\int_{A_k}\abs{f}^{p_ *}\right)^{\frac{1}{p_*}} \left(\int_{A_k}\left|G_k(u)\right|^{p^*}\right)^{\frac{1}{p^*}} .
    \end{equation}
    If $G_k(u)=0$ a.e., the proof is complete. Otherwise, we simplify equal terms to get
    \begin{equation}\label{proof:stimaLinfty:eq1}        
    \frac{\alpha}{\mathcal{S}_p^p} \left(\int_{A_k}\left|G_k(u)\right|^{p^*}\right)^\frac{p-1}{p^*} 
    \leq
    \left(\int_{A_k}\abs{f}^{p_ *}\right)^{\frac{1}{p_*}}.
    \end{equation}
    Since $$\frac{1}{p_*}\frac{p^*}{p-1}=\frac{Q+p'}{Qp'}\frac{Qp}{Q-p}\frac{1}{p-1}=\frac{Q+p'}{Q-p},$$
    we get, by rearranging the exponents in \eqref{proof:stimaLinfty:eq1}, 
    $$
    \int_{A_k}\left|G_k(u)\right|^{p^*}
    \leq
    C \left(\int_{A_k}\abs{f}^{p_ *}\right)^\frac{Q+p'}{Q-p}
    $$
    for some positive constant $C$ depending on $p,\alpha,Q$ and $\mathcal{S}_p$.
    Since $f$ belongs to $\elleom m$, with $m>\frac{Q}{p}\geq p_*=\frac{Qp'}{Q+p'}$, we use the H\"older inequality once again to get
    \begin{equation}\label{proof:stimaLinfty:eq2}  
        \int_{A_k}\left|G_k(u)\right|^{p^*} \leq C\left(\|f\|_{\elleom m}^{p_*} \abs{A_k}^{1-\frac{p_*}{m}} \right)^\frac{Q+p'}{Q-p}.
    \end{equation}
    We now want to apply Lemma \ref{LemmaStampacchiaLinfty} to the function $\psi(k)=\abs{A_k}$: let $h>k$, so that $A_h \subseteq A_k$, and $G_k(u) \geq h-k$ on $A_h$. By \eqref{proof:stimaLinfty:eq2}, we get
    $$
    (h-k)^{p^*} \abs{A_h} \leq
    C\left(\|f\|_{\elleom m}^{p_*} \abs{A_k}^{\frac{m-p_*}{m}} \right)^\frac{Q+p'}{Q-p},
    $$
    which implies
    $$
    \abs{A_h} \leq C \|f\|_{\elleom m}^{{p_*}\frac{Q+p'}{Q-p}} \frac{\abs{A_k}^{\frac{m-p_*}{m}\frac{Q+p'}{Q-p}}}{(h-k)^{p^*}} .
    $$
    Now we define 
    $$
    M=C \|f\|_{\elleom m}^{{p_*}\frac{Q+p'}{Q-p}}, \quad \gamma=p^*, \quad \delta=\frac{m-p_*}{m}\frac{Q+p'}{Q-p}
    $$
    and observe that the assumption $m>\frac Q p$ is equivalent to $\delta>1$. Thus $\psi(k)=\abs{A_k}$ satisfies the assumptions of Lemma \ref{LemmaStampacchiaLinfty}.
    
    We thus get that $\abs{A_d}=0$, where
    $$
    d=C(Q, p, \alpha,m,\mathcal{S}_p) \|f\|_{\elleom m}^{{p_*}\frac{Q+p'}{Q-p}\frac{1}{p^*}}.
    $$
    Since $\abs{A_d}=0$, we have $|u| \leq d$ a.e., which is
    $$
    \|u\|_{\elle{\infty}(\Omega)} \leq d=C(Q, p, \alpha,m,\mathcal{S}_p)\|f\|_{\elleom m}^\frac{1}{p-1},
    $$
    thereby concluding the proof.
\end{proof}

\begin{proof}[Proof of Theorem \ref{theo:Stampacchia-p**}]
    Let $k\geq0$, $\gamma>1$ and choose $$T_k (u) \abs{T_k (u)}^ {p(\gamma-1)}$$
    as a test function in \eqref{Problema_solof}, leading to
    $$[p(\gamma-1)+1]\io a(x,u, X  u)\cdot X T_k(u)\abs{T_k (u)}^ {p(\gamma-1)}\leq
    \io \abs{f(x)}\abs{T_k (u)}^ {p(\gamma-1)+1}.$$
    By \Cref{ass:ellipticity}, we have
    $$\alpha [p(\gamma-1)+1] \io \abs{ X  T_k(u)}^p\abs{ T_k(u)}^{(\gamma-1)p}\leq
    \io \abs{f(x)}\abs{T_k (u)}^ {p(\gamma-1)+1},$$
    which is
    $$\alpha \frac{[p(\gamma-1)+1]}{\gamma^p} \io \abs{ X  T_k(u)^\gamma}^{p}
    \leq
    \io \abs{f(x)}\abs{T_k (u)}^ {p(\gamma-1)+1}.$$
    Applying \Cref{SobolevINEQ}, we get
    $$\alpha \frac{[p(\gamma-1)+1]}{\mathcal{S}_p^p\gamma^p} \left(\io \abs{T_k(u)}^{\gamma p^*}\right)^\frac{p}{p^*}
    \leq
    \io \abs{f(x)}\abs{T_k (u)}^ {p(\gamma-1)+1},$$
    thus, applying the H\"older inequality to the right-hand side, we obtain:
    $$\alpha \frac{[p(\gamma-1)+1]}{\mathcal{S}_p^p\gamma^p} \left(\io \abs{T_k(u)}^{\gamma p^*}\right)^\frac{p}{p^*}
    \leq \norma{f}{\elleom m}\left(\io\abs{T_k (u)}^ {[p(\gamma-1)+1]m'}\right)^\frac{1}{m'}.$$
    
    Now we choose $\gamma$ such that
    $$\gamma p^*=[p(\gamma-1)+1]m',$$
    which is
    $$\gamma=\frac{m(Q-p)}{p'(Q-mp)}.$$
    Note that $\gamma\geq 1$ if and only if $m\geq p_*=\frac{Qp'}{Q+p'}$, which is our assumption. Moreover, we have
    $$\gamma p^*=\frac{(p-1)mQ}{Q-mp}=[m^*(p-1)]^*.$$
    This leads to
    $$\alpha \frac{[p(\gamma-1)+1]}{\mathcal{S}_p^p\gamma^p} 
    \left(\io \abs{T_k(u)}^{[m^*(p-1)]^*}\right)^{\frac{p}{p^*}-\frac{1}{m'}}
    \leq \norma{f}{\elleom m},$$
    which, since $\frac{p}{p^*}-\frac{1}{m'}=\frac{p-1}{[m^*(p-1)]^*}$, implies
    $$\norma {T_k(u)}{\elleom{[m^*(p-1)]^*}}\leq C\norma{f}{\elleom m}^\frac{1}{p-1}. $$
    It now suffices to apply Fatou's Lemma as $k\to\infty$ to conclude the proof.    
\end{proof}

Now, we consider the problem 
\begin{equation}\label{Problema_solodiv}
    \begin{dcases}
        X^*a(x,u, X u) = X^*F(x) \quad &\text{in } \Omega\\
        u=0 \quad &\text{on } \partial\Omega,
    \end{dcases}
\end{equation}
with $F\in\left(\elleom{m}\right)^d$, $m\geq p'$, we have the following two results. For the sake of brevity, we omit their proofs, being identical to the ones of, respectively, \Cref{Theo:u_limitata} and \Cref{theo:Stampacchia-p**}. In this case, to control the contribution of $Xu$ on the right-hand side, one should first apply the H\"older inequality, simplify the terms and then apply \Cref{SobolevINEQ} to the left-hand side.

\begin{theorem}\label{Theo:u_limitata_div}
    Let $F\in\left(\elleom{m}\right)^d$, with $m>\frac{Q}{p-1}$, and let $u$ be a weak solution to \eqref{Problema_solodiv}. Then $u\in\elleom\infty$ and there exists a constant $C=C(Q,p,\Omega,\alpha,m,X)$ such that
    $$
    \norma{u}{\elleom\infty} \leq C \norma F{\elleom m}^\frac{1}{p-1}.
    $$ 
\end{theorem}

\begin{theorem}\label{theo:Stampacchia-p**_div}
    Let $F\in\left(\elleom{m}\right)^d$, with $p'\leq m<\frac{Q}{p-1}$ and let $u$ be a weak solution to \eqref{Problema_solodiv}. Then $u\in\elleom {q}$, with 
    \[q=\frac{m(p-1)Q}{Q-m p}=[m(p-1)]^*,\]
    and there exists a constant $C=C(Q,p,\Omega,\alpha,m,X)$ such that
    $$
    \norma{u}{\elleom q} \leq C \norma F{\elleom m}^\frac{1}{p-1}.
    $$ 
\end{theorem}

\begin{remark}
    In \cite{Pic-Xell-linear}, the existence and uniqueness of a suitable class of solutions to the linear elliptic problem
    $$X^*(A(x)Xu)=f(x),$$
    for general $L^1$ (or measure) data was proved. In this paper, we have only discussed the existence of a solution for data belonging to the dual space of $\SobomX p$. However, we point out that the inequality
    $$m^*(p-1)\geq 1,$$
    which is required for the exponent $[m^*(p-1)]^*$ to be well-defined,
    can hold under much more relaxed assumptions than $m\geq p_*$. This hints that the theory of solutions found by approximation (see \cite{pLaplacian-BoccardoMerida}, \cite{BG-rhs_measures}, \cite{BoccardoGallouet}) and the one of entropy solutions to elliptic equations (\cite{LavoroB6}) may be extended to this subelliptic setting.
\end{remark}

\begin{remark}
    Note that the estimates we proved in this section \say{become worse} as $Q$ increases. More precisely, while $Q$ may be chosen to be arbitrarily large in \eqref{QDoubling}, better results are obtained for smaller values of $Q$. This is consistent with the fact that, in Riemannian Carnot Groups, the optimal $Q$ increases as the number of steps needed to satisfy the H\"ormander condition increases (see \cite[Proposition 11.15]{HK-SobPoinc}), which, in turn, leads to a more degenerate second order operator.
    In particular, in the uniformly elliptic case (when $X=\nabla$ and $Q=N$), the best estimate is given by Theorem 2.1 of \cite{pLaplacian-BoccardoMerida}, which is known to be sharp.
\end{remark}

\subsection{The case \texorpdfstring{$p=Q$}{p=Q}}\label{Sec:Regularity_p=Q}

While the assumption $p>Q$ immediately yields, by \Cref{MorreyTH}, the H\"older continuity of the solution $u$ to \eqref{Problema_solof} (provided that the datum $f$ belongs, at least, to $\elleom 1$), the borderline case $p=Q$ requires a separate analysis. In this setting, \Cref{TrudingerINEQ} ensures that any function in $\SobomX Q$ has exponential integrability. 

Moreover, under the assumption that $f \in \elleom m$ for some $m>1$, we can prove the existence and the boundedness of the solution, as shown by the following theorem.

\begin{theorem}\label{theo:p=Q:bounded}
    Let $p=Q$. Let $f$ belong to $\elleom m$ for some $m>1$. Then there exists a weak solution $u$ to \eqref{Problema_solof}. Moreover, $u\in\elleom\infty$.
\end{theorem}
\begin{proof}
    First, we point out that $f$ belongs to the dual of $\SobomX Q$. This follows from the fact that $\SobomX Q$ continuously embeds into $\elleom q$ for every $q<\infty$. The existence of $u$ thus follows from Theorem \ref{LL-Xelliptic}.
    Now let $k>0$ and choose $G_k(u)$ as a test function in \eqref{Problema_solof}. Remark \ref{Rem:TkGk} and \Cref{ass:ellipticity} lead to:
    $$\alpha\io \abs{XG_k(u)}^Q\leq \io a(x,u,Xu)\cdot XG_k(u) \leq\int_{\abs{u}\geq k}\abs{f(x)}\abs{G_k(u)}.$$
    Let $p$ be such that $1<p<Q$ and $m>p_*$: by the Sobolev embedding and the H\"older inequality, we have
    $$\norma{G_k(u)}{\elleom{p^*}}\leq C\norma{X G_k(u)}{\elleom p}\leq C\norma{X G_k(u)}{\elleom Q},$$
    which means
    $$\left(\io \abs{G_k(u)}^{p^*}\right)^{\frac{Q}{p^*}}\leq \alpha C \io \abs{XG_k(u)}^Q \leq C \left(\int_{\abs{u}\geq k}\abs{f(x)}^{p_*}\right)^\frac{1}{p_*} \norma{G_k(u)}{\elleom {p^*}}.$$

    $$\left(\io \abs{G_k(u)}^{p^*}\right)^{\frac{Q-1}{p^*}}\leq C \left(\int_{\abs{u}\geq k}\abs{f(x)}^{p_*}\right)^\frac{1}{p_*}\leq 
    C\norma{f}{\elleom m} \abs{\{\abs{u}\geq k\}}^{\frac{1}{p_*}-\frac{1}{m}}$$
    Let $h>k>0$ and define $\psi(k)= \abs{\{\abs{u}\geq k\}}$. Repeating the steps of the proof of Theorem \ref{Theo:u_limitata}, we obtain:
    $$(h-k)^{Q-1} \psi(h)^\frac{Q-1}{p^*}\leq C\norma{f}{\elleom m} \psi(k)^{\frac{1}{p_*}-\frac{1}{m}},$$
        that is
    $$\psi(h)\leq C\frac{\psi(k)^{\frac{p^*}{Q-1}\left(\frac{1}{p_*}-\frac{1}{m}\right)}}{(h-k)^{p^*}}.$$
    To apply Lemma \ref{LemmaStampacchiaLinfty}, we only need to check that $\frac{p^*}{Q-1}\left(\frac{1}{p_*}-\frac{1}{m}\right)>1$. This, however, is always possible, since 
    $$\lim _{p\to Q^-}\frac{p^*}{Q-1}\left(\frac{1}{p_*}-\frac{1}{m}\right) =+\infty.$$
    It follows that there exists $d>0$ such that $\psi(d)=0$, thus $\norma{u}{\elleom\infty}<+\infty.$    
\end{proof}

Analogously, if we consider the problem \eqref{Problema_solodiv}, we have the following result.
\begin{theorem}
    Let $F\in\left(\elleom m\right)^d$ for some $m>Q'$ and let $u$ be a weak solution to \eqref{Problema_solodiv} (which exists by \Cref{LL-Xelliptic}). Then $u\in\elleom\infty$.
\end{theorem}
\begin{proof}
    This proof follows the procedure used to prove \Cref{Theo:u_limitata_div}, with a slight alteration in the use of the Sobolev embedding, as in the proof of \Cref{theo:p=Q:bounded}.
    Let $k>0$ and choose $G_k(u)$ as a test function in \eqref{Problema_solodiv}. Remark \ref{Rem:TkGk} and \Cref{ass:ellipticity} lead to:
    $$\alpha\io \abs{XG_k(u)}^Q\leq \io a(x,u,Xu)XG_k(u) \leq \io F(x)\cdot XG_k(u)$$
    which, using the Young inequality on the right-hand side, yields
    \begin{equation}\label{proof:p=Q:boundedDIV:1}
        \alpha\io \abs{XG_k(u)}^Q \leq C\int_{\abs{u}\geq k}\abs{F(x)}^{Q'} +\frac{\alpha}{2}\io \abs{XG_k(u)}^Q
    \end{equation}
    for some positive constant $C$ depending on $\alpha$ and $Q$.
    As before, we choose $p\in(1,Q)$. Then, by \Cref{SobolevINEQ} and the H\"older inequality, we have
    $$\norma{G_k(u)}{\elleom{p^*}}\leq \mathcal{S}_p\norma{X G_k(u)}{\elleom p}\leq C(\mathcal{S}_p,\Omega,p,Q)\norma{X G_k(u)}{\elleom Q},$$
    which, together with \eqref{proof:p=Q:boundedDIV:1}, leads to
    $$\left(\io \abs{G_k(u)}^{p^*}\right)^\frac{Q}{p^*}\leq C \int_{\abs{u}\geq k}\abs{F(x)}^{Q'}$$
    for some positive constant $C=C(\mathcal{S}_p,\Omega,p,Q,\alpha)$.
    Next, we apply the H\"older inequality to the last term with exponent $\frac{m}{Q'}$
    $$\left(\io \abs{G_k(u)}^{p^*}\right)^\frac{Q}{p^*}\leq C \norma{F}{\elleom m}^{Q'}\abs{\{\abs{u}\geq k\}}^{1-\frac{Q'}{m}}.$$
    Let $h>k>0$ and define $\psi(k)= \abs{\{\abs{u}\geq k\}}$. As before, we obtain:
        $$(h-k)^Q \psi(h)^\frac{Q}{p^*}\leq C \norma{F}{\elleom m}^{Q'} \psi(k)^{1-\frac{Q'}{m}},$$
        that is
        $$\psi(h)\leq C\frac{\psi(k)^{\left(1-\frac{Q'}{m}\right)\frac{p^*}{Q}}}{(h-k)^{p^*}}.$$
        To apply Lemma \ref{LemmaStampacchiaLinfty}, we need to check that $\left(1-\frac{Q'}{m}\right)\frac{p^*}{Q}>1$. This, however, is always possible, since $1-\frac{Q'}{m}\in(0,1)$ and $$\lim _{p\to Q^-}\frac{p^*}{Q} = \lim _{p\to Q^-}\frac{p}{Q-p}\to+\infty.$$
        It follows that there exists $d>0$ such that $\psi(d)=0$, thus $\norma{u}{\elleom\infty}<+\infty.$   
\end{proof}

\begin{remark}
    Any measurable function $f$ such that $f(x)\log(1+\abs{f(x)})\in\elleom1$ belongs to the dual of $\SobomX Q$. 
\end{remark}
\begin{proof}    
    Using the Young inequality
    \begin{equation}
        ab\leq a\log(1+a)+e^b\quad\forall a,b\geq0
    \end{equation}
    to estimate $\io vf(x)$ for some $v\in\SobomX Q$, we have
    \begin{equation}
        \begin{split}
            \abs{\io vf(x)}&\leq \frac{1}{C_1}\norma{Xv}{\elleom Q}\io C_1\frac{\abs{v}}{\norma{Xv}{\elleom Q}}\abs{f(x)}\\
            &\leq \left(\io \abs{f(x)}\log(1+\abs{f(x)}) + \io e^{C_1\frac{\abs{v}}{\norma{Xv}{\elleom Q}}}\right),
        \end{split}
    \end{equation}
    where $C_1$ is the constant appearing in \Cref{TrudingerINEQ}.
    Now we apply \Cref{TrudingerINEQ} as follows: for simplicity, we define $y=C_1\frac{\abs{v}}{\norma{Xv}{\elleom Q}}$
    $$\io e^{y}\leq \int_{y\leq 1} e^{y} + \int_{y\geq 1} e^{y}\leq e\abs{\Omega} + \int_{y\geq 1} e^{y^\frac{Q}{Q-1}}
    \leq e\abs{\Omega} + \io e^{y^\frac{Q}{Q-1}}\leq e\abs{\Omega} + C_T$$
    we thus have
    $$\abs{\io vf(x)}\leq{\norma{Xv}{\elleom Q}}\left(\norma{f\log(1+\abs{f})}{\elleom1}+e\abs{\Omega}+C_T\right),$$
    which proves that $f\in\left(\SobomX Q\right)^\prime.$
\end{proof}
        
    We conclude by pointing out that, as was shown in \cite{Boccardo-NLapl}, the requirement 
    \[f\log(1+\abs{f})\in\elleom1\]
    is not sufficient to ensure the boundedness of the solution $u$ to
    \begin{equation}
    \begin{dcases}
        -\operatorname{div} (\abs{\nabla u}^{N-2}\nabla u)=f(x) \quad &\text{in } \Omega\\
        u=0 \quad &\text{on } \partial\Omega.
    \end{dcases}
\end{equation}
    whenever $N>2$. The optimal assumptions on $f$ under which the solution $u$ to \eqref{Problema_solof} belongs to $\elleom\infty$ are still unclear, even in the uniformly elliptic case.
    For related results in dimension $2$, we refer the reader to \cite{AlbericoFerone}, where symmetrization techniques are employed to establish sufficient conditions for the boundedness of solutions to the Laplace equation.

\section*{Acknowledgements}
I thank Luigi Orsina for his valuable feedback and his help in revising this manuscript.
I am also grateful to Marco Bramanti and Ermanno Lanconelli for kindly answering my questions regarding the existing literature on $X$-elliptic operators.\\
I declare no conflicts of interest. \\
I am a member of the GNAMPA group of INdAM.

\end{document}